\documentclass[11pt, reqno]{amsart}

\usepackage[linguistics]{forest}
\usetikzlibrary{trees}

\usepackage{latexsym}
\usepackage{amssymb}
\usepackage{mathrsfs}
\usepackage{amsmath,stackengine}
\usepackage{fancybox,color}
\usepackage{enumerate}
\usepackage{hyperref}
\usepackage{eurosym}
\usepackage[utf8]{inputenc}

\newcommand{\mmd}{ \mathrm{d}}
\newcommand{\dy}{\, \mathrm{d}y}

\newcommand{\s}{\star}

\newcommand{\kom}[1]{}
\renewcommand{\kom}[1]{{\bf [#1]}}

\numberwithin{equation}{section}

\usepackage{geometry}
 \def\1{\raisebox{2pt}{\rm{$\chi$}}}

\newtheorem{theorem}{Theorem}[section]
\newtheorem{corollary}[theorem]{Corollary}
\newtheorem{lemma}[theorem]{Lemma}
\newtheorem{proposition}[theorem]{Proposition}
\newtheorem{fact}[theorem]{Fact}
\newtheorem{claim}[theorem]{Claim}
\newtheorem{remark}[theorem]{Remark}

\newtheorem{conjecture}[theorem]{Conjecture}

\newcommand{\R}{{\mathbb R}}

\newcommand{\N}{{\mathbb N}}
\newcommand{\Z}{{\mathbb Z}}
\newcommand{\C}{{\mathcal C}}

 \newcommand{\eps}{{\varepsilon}}
 \def\1{\raisebox{2pt}{\rm{$\chi$}}}
 
\def\vint_#1{\mathchoice%
          {\mathop{\kern 0.2em\vrule width 0.6em height 0.69678ex depth -0.58065ex
                  \kern -0.8em \intop}\nolimits_{\kern -0.4em#1}}%
          {\mathop{\kern 0.1em\vrule width 0.5em height 0.69678ex depth -0.60387ex
                  \kern -0.6em \intop}\nolimits_{#1}}%
          {\mathop{\kern 0.1em\vrule width 0.5em height 0.69678ex
              depth -0.60387ex
                  \kern -0.6em \intop}\nolimits_{#1}}%
          {\mathop{\kern 0.1em\vrule width 0.5em height 0.69678ex depth -0.60387ex
                  \kern -0.6em \intop}\nolimits_{#1}}}
\def\vintslides_#1{\mathchoice%
          {\mathop{\kern 0.1em\vrule width 0.5em height 0.697ex depth -0.581ex
                  \kern -0.6em \intop}\nolimits_{\kern -0.4em#1}}%
          {\mathop{\kern 0.1em\vrule width 0.3em height 0.697ex depth -0.604ex
                  \kern -0.4em \intop}\nolimits_{#1}}%
          {\mathop{\kern 0.1em\vrule width 0.3em height 0.697ex depth -0.604ex
                  \kern -0.4em \intop}\nolimits_{#1}}%
          {\mathop{\kern 0.1em\vrule width 0.3em height 0.697ex depth -0.604ex
                  \kern -0.4em \intop}\nolimits_{#1}}}

\newcommand{\aveint}[2]{\mathchoice%
          {\mathop{\kern 0.2em\vrule width 0.6em height 0.69678ex depth -0.58065ex
                  \kern -0.8em \intop}\nolimits_{\kern -0.45em#1}^{#2}}%
          {\mathop{\kern 0.1em\vrule width 0.5em height 0.69678ex depth -0.60387ex
                  \kern -0.6em \intop}\nolimits_{#1}^{#2}}%
          {\mathop{\kern 0.1em\vrule width 0.5em height 0.69678ex depth -0.60387ex
                  \kern -0.6em \intop}\nolimits_{#1}^{#2}}%
          {\mathop{\kern 0.1em\vrule width 0.5em height 0.69678ex depth -0.60387ex
                  \kern -0.6em \intop}\nolimits_{#1}^{#2}}}

\begin{document}

\title{On optimal autocorrelation inequalities on the real line}
\author{Jos{\'e} Madrid and Jo\~ao P. G. Ramos}

\date{\today}
\subjclass[2010]{.}

\address{José Madrid: Department of  Mathematics,  University  of  California,  Los  Angeles (UCLA),  Portola Plaza 520, Los  Angeles,
California, 90095, USA}
\email{jmadrid@math.ucla.edu}		

\address{Jo\~ao P.G. Ramos: Instituto Nacional de Matem\'atica Pura e Aplicada (IMPA), Estrada Dona Castorina 110, Rio de Janeiro, RJ, 22460-320, Brazil}
\email{jpgramos@impa.br}


\maketitle
\begin{abstract}
We study 
autocorrelation inequalities, in the spirit of Barnard and Steinerberger's work \cite{BarnardSteinerberger}. In particular, we obtain improvements on the sharp constants in some of the inequalities previously considered by these authors, and also 
prove existence of extremizers to these inequalities in certain specific settings. Our methods consist of relating the inequalities in question to other classical sharp inequalities in Fourier analysis, such as the sharp Hausdorff--Young inequality, and employing functional analysis as well as 
measure theory tools in connection to a suitable dual version of the problem to identify and impose conditions on extremizers. 
\end{abstract}

\section{Introduction}
The study of auto-convolution and auto-correlation inequalities in the real line has attracted the attention of many authors in the last few years. Indeed, since the results by Cilleruelo, Ruzsa and Vinuesa \cite{CRV} connecting the problem of finding the best constant $c>0$ so that 
\[
\text{max}_{-1/2 \le t \le 1/2} \int_{\R} f(t-x)f(x) \, \mmd x \ge c \left(\int_{-1/4}^{1/4} f(x) \, \mmd x\right)^2,
\]
for all $f \in L^1(\R)$ supported in $[-1/4,1/4],$ to the asymptotic size of $g-$Sidon sets, many authors have made an attempt to find the best $c>0$ above. Recent progress on this question can be found in \cite{CRT,G,MOB1,Y,MOB2,MV} and, more recently, in \cite{CS}, where the authors prove that $c \ge 1.28.$ This, however, is still relatively far off 
from the best upper bound, $c \le 1.52$, proven by Matolcsi and Vinuesa \cite{MV}. 

In a recent manuscript, Barnard and Steinerberger \cite{BarnardSteinerberger} have considered two other inequalities related to combinatorics and number theory problems. In fact, it was proved in \cite{BarnardSteinerberger}
that the following inequality about the mean value of the autocorrelation of a function $f \in L^1(\R) \cap L^2(\R)$ holds:
\begin{equation}\label{eq:l1l2}
\int_{-1/2}^{1/2}\int_{\R}f(x)f(x+t) dx dt\leq 0.91\|f\|_1\|f\|_2.
\end{equation}
It was also proved in \cite{BarnardSteinerberger}
that the following inequality regarding, this time, the \emph{minimum} value of the autocorrelation holds
\begin{equation}\label{eq:minimum-first} 
\min_{t\in{[0,1]}}\int_{\R}f(x)f(x+t) dx \leq
\frac{1}{2(1+\theta_0)}\|f\|^2_1
\end{equation}
for any function $f\in L^{1}(\R)$, where \[
\theta_0 := - \inf_{x\in\R}\frac{\sin(x)}{x} = 0.217\dots 
\]
Our main goal, in this manuscript, is to further explore inequalities \eqref{eq:l1l2} and \eqref{eq:minimum-first}. In doing so, we improve the best constant in \eqref{eq:l1l2} from 0.91 to 0.87, and manage to prove that the best constant in \eqref{eq:minimum-first} is, under mild additional assumptions on the class of functions considered, 
\emph{strictly} smaller than $\frac{1}{2(1+\theta_0)}.$ Our methods of proof for these results, however, are quite different between themselves: for the former inequality, we find a new approach to the problem of finding better constants, relating it to the Hausdorff--Young inequality in a suitable way, whereas for the latter our methods are heavily based on a careful analysis of extremal functions to the problem. 

This leads us naturally to distinguish our results into two kinds: the ones which, as the proof of our improvement to \eqref{eq:l1l2}, are more quantitative in nature, and the ones, as our argument to do better than \eqref{eq:minimum-first}, are more qualitative. 

\subsection{Quantitative results}

Our first result concerns the mean of the auto-correlation of a function in $L^{1}(\R)\cap L^{2}(\R)$.
\begin{theorem}\label{Thm 1}
For any $f\in L^{1}(\R)\cap L^{2}(\R)$. The following inequality holds 
$$
\int_{-1/2}^{1/2}\int_{\R}f(x)f(x+t) dx dt\leq 0.864\|f\|_1\|f\|_2,
$$
and $0.864$ cannot be replaced by 0.8. 
\end{theorem}


The lower bound was previously established (through an example) by Barnard and Steinberger in \cite{BarnardSteinerberger}.\\

The strategy to prove \ref{Thm 1} is broad in the sense that it can also be applied to obtain estimates and existence of extremizers for the integral of the autocorrelation
with some other probability measures. Because of the nature of the tools used, it is natural to consider Gaussian means instead of interval averages on the left hand side. Our next Theorem refers to that situation.

\begin{theorem}\label{Thm 3}
Let $a$ be a positive real number. For any $f\in L^{1}(\R)\cap L^{2}(\R)$. The following inequality holds
\begin{align*}
\left(\frac{a}{\pi}\right)^{1/2}\int_{\R}\int_{\R}f(x)f(x+t)e^{-a t^{2}}dxdt\leq \left(\frac{8a}{27\pi}\right)^{1/4}\|f\|_1\|f\|_2,
\end{align*}
and $\left(\frac{8a}{27\pi}\right)^{1/4}$ can not be replaced by $\left(\frac{a}{4\pi}\right)^{1/4}$. 
\end{theorem}
In particular, if $a=2\pi$ our upper bound is 0.8773 and our lower bound 0.8408.\\

In Theorems \ref{Thm 1} and \ref{Thm 3}, the strategy, as previously mentioned, is of relating our results to a dual problem involving the Fourier transform, and then employing some well-known sharp inequality. In the case of these two results, we will
employ the sharp Hausdorff--Young inequality in a suitable way, and run an optimization process to improve the constant in the end. \\

Lastly, we consider the second type of problems described in the introduction. In this case we take the minimum of the auto-convolution in an interval instead of the average.
As a consequence of Theorem \ref{Thm 1} trivially we obtain that for any $f\in L^{1}(\R)\cap L^{2}(\R)$, the following inequality holds
\begin{equation}\label{Main eq 1}
\min_{t\in[{-1/2,1/2}]}\int_{\R}f(x)f(x+t)dxdt\leq 0.8641\|f\|_1\|f\|_2.
\end{equation}
However, naturally we can expect to get a better bound, that is the content of our next theorem.
\begin{theorem}\label{thm min 1}
For any $f\in L^{1}(\R)\cap L^{2}(\R)$. The following inequality holds
\begin{equation}\label{Min eq 1}
\min_{t\in[{-1/2,1/2}]}\int_{\R}f(x)f(x+t)dxdt\leq 0.829604\|f\|_1\|f\|_2.
\end{equation}
Moreover, the constant $0.829604$ can not be replaced by $0.544$.
\end{theorem}

The proof of Theorem \ref{thm min 1} follows closely the lines of the proof of \eqref{eq:minimum-first}, as in \cite{BarnardSteinerberger}. The main difference is that now we interpolate their olds bounds with the strategy of proof from Theorems \ref{Thm 1} and \ref{Thm 3}, 
and this allows us to improve the constant marginally, in the presence of the mix between $L^1$ and $L^2$ norms.  

The lower bound in Theorem \ref{thm min 1} is not surprising, we include this for completeness. Numerical and computational methods might be useful to find examples generating better lower bounds, we plan to explore this in a future project.

\subsection{Qualitative results}

In order to introduce the qualitative results, we let, for a given $g \in L^1(\R), g \ge 0,$ and $I \subset \R$ compact, $\mathcal{L}_g(I)$ denote the class of nonnegative $L^1$ functions $f$ on $\R$, such that $f \le g$ \emph{outside}  $I.$

The definition of $L_g(I)$ might seem a little arbitrary, as it does not in general, as far as the authors know, coincide with any definition of a classical function space. However, we remark on an intuition stemming from \cite{BarnardSteinerberger}, in form of a conjecture, which helps understand the introduction of this additional concept:

\begin{conjecture}\label{Conj Compact} There exist extremizers to \eqref{Eq Extremizer 1} and \eqref{eq min}. Moreover, those extremizers can be taken to have compact support. 
\end{conjecture}

Indeed, one readily notices that the class of nonegative, integrable and compactly supported functions on a given interval $[-M,M]$ can be described as $\mathcal{L}_0([-M,M]).$ On the other hand, the definition of $L_g(I)$ turns out to be crucial for our purposes, as it enforces simultaneous uniform decay on our class of functions. 

Our first qualitative result establishes the existence of extremizers for \eqref{Main eq 1} when we replace $0.8641$ by the optimal constant, and whenever we are restricted to a certain class $\mathcal{L}_g(I).$ 

\begin{theorem}\label{Thm 2}
Let $C_{opt}$ the optimal constant in \eqref{Main eq 1} under the constraint of belonging to $\mathcal{L}_g(I),$ for some $g \ge 0$ and some $I \subset \R$ compact. That is, the smallest constant so that $0.8641$ can be replaced by it in \eqref{Main eq 1}, for all $f \in \mathcal{L}_g(I).$

Then there exists a function $f \in L^{1}(\R)\cap L^{2}(\R) \cap \mathcal{L}_g(I)$ such that
\begin{equation}\label{Eq Extremizer 1}
\int_{-1/2}^{1/2}\int_{\R}g(x)g(x+t)dxdt= C_{opt}\|g\|_1\|g\|_2.
\end{equation}
\end{theorem}

\begin{remark}
We can also establish the existence of extremizers for the Gaussian means problem following the lines in the proof of Theorem \ref{Thm 2}.
\end{remark}

The proof of these results uses functional analysis methods. Indeed, the first task is to identify a suitable formulation to this problem involving the Fourier transform. After doing that, we prove that extremizing sequences must converge, in a weak sense, to a certain function,
a property also satisfied by their Fourier transforms and the squares of their Fourier transforms. In the end, Fatou's Lemma and a careful analysis of the functions involved allows us to conclude. \\

Finally, we address another inequality previously approached by Barnard and Steinberger. This involves the minimum over an interval of the autocorrelation function, in comparison to the $L^1$ norm squared. In contrast
to our previous results, we cannot obtain an \emph{effective} result, such as an explicit bound that lowers the best constant, but we can only prove that the best constant in such a result is strictly lower than 
the one previously obtained. 

\begin{theorem}\label{Thm 4}
Let $C_4>0$ be the smallest constant such that the following inequality
\begin{equation}\label{eq min}
\min_{t\in[0,1]}\int_{\R}f(x)f(t+x)dx\leq C_4\|f\|_1^{2}
\end{equation}
holds for any $f\in L^{1}(\R) \cap \mathcal{L}_g(I),$ with $g,I$ as in Theorem \ref{Thm 2} above. Let $y_0$ be the smallest positive number in the set $\{ y \in \R \colon y = \tan (y)\},$ and define 
\[
\theta_0 := - \frac{\sin(y_0)}{y_0} = 0.217\dots 
\]
Then it holds that $C_4$ is \emph{strictly} smaller than $\frac{1}{2(1+\theta_0)}.$ 
\end{theorem}
In fact, as we  do not possess a device telling us how to iteratively construct an extremizer to \eqref{eq min}, we cannot dream of quantifying the best constant in Theorem \ref{Thm 4}. Nevertheless, the main message of 
this result is not the effective bound it gives, but the underlying message: in order to improve over the previous result, a new method altogether is needed, and maybe a strategy producing a measure that optimizes \eqref{eq min} 
would come in handy. In fact, a corollary of the proof of Theorem \ref{Thm 4} gives us the following. 

\begin{corollary}\label{Thm 5} There exists a finite, positive measure $\mu_0$ that maximizes the quantity 
\begin{equation}\label{eq min measure} 
\min_{1 > a > b > 0} \frac{ \mu \star \mu ( [b,a])}{(a-b) \|\mu\|_{TV}^2} =: \tilde{C}_4
\end{equation}
over the class of all finite, positive measures $\mu$ on the real line, such that $\mu_0|_{\R \setminus I}$ is absolutely continuous with respect to the Lebesgue measure and pointwise bounded by $g(x) 1_I(x)$. (here, $\mu \star \mu (A) = \int_{\R \times \R} 1_A(x-y) \, \mmd \mu(x) \, \mmd \mu (y)$ denotes autocorrelation). 
\end{corollary}

Here, we let $\|\mu\|_{TV} := |\mu|(\R)$ denote the \emph{total variation} of the finite measure $\mu$. Corollary \ref{Thm 5} is a compromise between what was known and what Conjecture \ref{Conj Compact} predicts: \emph{given} that we are working on an almost-compactly supported setting, then we have extremizers, although we might need to go to a larger class than $L^1$ functions. 

Indeed, if $\mmd \mu = f(x) \mmd x$ is absolutely continuous, then the definitions imply
\[
\frac{\mu \s \mu ([b,a])}{a-b} = \frac{1}{a-b} \int_b^a f \s f (t) \mmd t,
\]
and also $\|\mu\|_{TV} = \|f\|_1.$ Thus, by the Lebesgue differentiation theorem,
\[
\min_{1 > a > b > 0} \frac{ \mu \star \mu ( [b,a])}{(a-b) \|\mu\|_{TV}^2} = \text{ess } \min_{t \in [0,1]} \frac{f\s f(t)}{\|f\|_1^2}.
\]
For this differentiation reason, we might use the alternate notation 
\[
\min_{1 > a > b > 0} \frac{ \mu \star \mu ( [b,a])}{(a-b) \|\mu\|_{TV}^2}  = \min_{1> t > \eps > 0} \frac{\mu \s \mu ([t-\eps,t])}{\eps \|\mu\|_{TV}^2}
\]
throughout proofs. \\

Theorem \ref{Thm 4} and its corollary are the lengthiest and perhaps the most technical in this manuscript. Instead of clean, direct proofs, mainly available for Theorems \ref{Thm 1}, \ref{Thm 3} and \ref{Thm 2}, we need to work hands-on to the task of finding an extremizer. This is achieved
by looking into a suitable Fourier-dualized version of the problem, together with functional analysis considerations and the Bochner theorem on positive definite functions. After knowing that there are extremizers - at least when we move past functions and consider positive measures instead -, 
we need to prove that they cannot be too singular. 

This, although technically stated in the proof, has a simple explanation: convolution makes, in general, smoother, and objects whose autoconvolution is too singular would have to have much worse behaviour than we allow them in our class of measures. 
This amounts to some measure theory considerations about autocorrelation of positive measures, and proves, in an abstract manner, that the constant by Barnard and Steinerberger for \eqref{eq:minimum-first} is not the optimal one, at least when we restrict to specific classes containing compactly supported functions.This improvement 
alone is new and, to the best of our knowledge, cannot be obtained directly from the techniques in \cite{BarnardSteinerberger}. 

It is interesting to note that Theorems \ref{Thm 2} and \ref{Thm 4} restrict to a given class of functions. Although we still believe that their conclusions should hold if one drops the assumption of belonging to the class $\mathcal{L}_G(I),$ they can still be regarded to be (almost) as interesting as these stronger versions, for two main reasons. 
The first, upon which we expand more in the final section of this manuscript, is that many of the extremizers to these problems - in case they exist - are conjectured to be also compactly supported, and thus the restriction on the classes of functions would be very superfluous, 
would this suspicion be confirmed. The second is that, as restrictive as they are, they still can provide us with interesting characterizations of how extremizers much behave. Indeed, the proof of Theorem \ref{Thm 4} implies the following: 

\begin{corollary} Considering the problem of maximizing \eqref{eq min measure}, at least one of the following holds: 
\begin{enumerate}
 \item either there is \emph{no} nonnegative measure $\mu_0$ that maximizes \eqref{eq min measure} over the class of all finite, positive measures $\mu$ on the real line;
 \item or the sharp constant $\mathfrak{C}_4$ over this class is \emph{strictly} less than than $\frac{1}{2(1+\theta_0)},$ with $\theta_0$ as in Theorem \ref{Thm 4}.
\end{enumerate}
Moreover, the only way the first item can happen is that \emph{no} extremizing sequence $\{\mu_n\}_{n \ge 0}$ is tight. 
\end{corollary}

It has recently come to our knowledge that, in the recent manuscript \cite{FKM}, the authors investigate properties that an extremizer to \eqref{eq min} must necessarily fulfill. Although they do not prove existence of extremizers, we believe that a suitable combination of their methods with ours may 
result in further progress towards lowering the constant in Theorem \ref{Thm 4} towards the best constant. 

Finally, a word on the notation: along this paper we will keep using the classical notation
for the complement of a set, the $L^p$ norm of a function $f:\R\to\R$ denoted by $\|f\|_p$ and the dual H\"older exponent of $p>1$ denoted by $p'$. Moreover,
for any  integrable function $f\in L^{1}(\R)$ we denote by $\widehat{f}$ its corresponding Fourier transform given by
$$
\widehat{f}(x):=\int_{\R}f(y)e^{-2\pi ixy}\dy.
$$

We may sometimes make use of the Schwartz class $\mathbb{S}(\R)$ of functions that, along with their derivatives, decay faster than the inverse of any polynomial. Sometimes we will make use of approximation arguments 
to prove our main results, by proving them first to $\mathbb{S}(\R).$ Due to the standard nature of these results, we omit them. 

\section{Proof of the quantitative results} 

We start by providing the proof of the results in which we effectively improve the previous constants. \\ 

\begin{proof}[Proof of Theorem \ref{Thm 1}]
For any $p> 1$, by Plancherel and the H\"older's inequality we have that
\begin{align*}
    \int_{-1/2}^{1/2}\int_{\R}f(x)f(x+t) dx dt&=\int_{\R}|\hat f(\xi)|^{2}\frac{\sin(\pi\xi)}{\pi\xi} d\xi\\
    &\leq \|\hat f\|^{2}_{2(p')}\left(\int_{\R}\left|\frac{\sin(\pi\xi)}{\pi\xi}\right|^{p}d\xi\right)^{\frac{1}{p}}.
\end{align*}
We observe that $1<\frac{2p}{p+1}<2$ for all $p>1$, then, using the optimal Hausdorff-Young inequality \cite{Beckner} and interpolating we see that
\begin{align*}
\|\hat f\|^{2}_{2(p')} &\leq \left(\frac{2p}{p+1}\right)^{\frac{p+1}{2p}}\left(\frac{2p}{p-1}\right)^{\frac{-(p-1)}{2p}}\|f\|^2_{\frac{2p}{p+1}}   \\
&\leq \frac{(2p)^{\frac{1}{p}}(p-1)^{\frac{p-1}{2p}}}{(p+1)^{\frac{p+1}{2p}}}\|f\|^{\frac{2}{p}}_1\|f\|^{2-\frac{2}{p}}_2.
\end{align*}

Combining these two estimates we obtain
\begin{align}
    &\int_{-1/2}^{1/2}\int_{\R}f(x)f(x+t) dx dt\nonumber\\
    &\leq \frac{(2p)^{\frac{1}{p}}(p-1)^{\frac{p-1}{2p}}}{(p+1)^{\frac{p+1}{2p}}}\left(\int_{\R}\left|\frac{\sin(\pi\xi)}{\pi\xi}\right|^{p}d\xi\right)^{\frac{1}{p}}\|f\|^{\frac{2}{p}}_1\|f\|^{2-\frac{2}{p}}_2\ \ \ \text{for every}\ \ p>1. \label{HY consequence}
    \end{align}

Using this and the trivial bound we conclude that
\begin{align*}
    &\int_{-1/2}^{1/2}\int_{\R}f(x)f(x+t) dx dt\\
    &\leq 
    \left(
    \left( \frac{(2p)^{\frac{1}{p}}(p-1)^{\frac{p-1}{2p}}}{(p+1)^{\frac{p+1}{2p}}}\left(\int_{\R}\left|\frac{\sin(\pi\xi)}{\pi\xi}\right|^{p}d\xi\right)^{\frac{1}{p}}\|f\|^{\frac{2}{p}}_1\|f\|^{2-\frac{2}{p}}_2  
    \right)^{p}\|f\|^{2(p-2)}_1 \right)^{\frac{1}{2(p-1)}}\\
    &=: C_p\|f\|_1\|f\|_2 \ \ \text{for all}\ \ p\geq2,
    \end{align*}
finally, we observe that $\inf_{p\geq 2} C_p=0.864$ which implies the desired result.
\end{proof}

\begin{proof}[Proof of Theorem \ref{Thm 3}]
We can follow the lines in the proof of Theorem \ref{Thm 1} to obtain that for every $p\geq 2$ the following inequality holds
\begin{align*}
    &\left(\frac{a}{\pi}\right)^{1/2}\int_{\R}\int_{\R}f(x)f(x+t)e^{-a t^{2}}dxdt\\
    &\leq 
    \left(
    \left( \frac{(2p)^{\frac{1}{p}}(p-1)^{\frac{p-1}{2p}}}{(p+1)^{\frac{p+1}{2p}}}\left(\int_{\R} e^{-\pi^2 p\xi^{2}/a}d\xi\right)^{\frac{1}{p}}\|f\|^{\frac{2}{p}}_1\|f\|^{2-\frac{2}{p}}_2  
    \right)^{p}\|f\|^{2(p-2)}_1 \right)^{\frac{1}{2(p-1)}}\\
    &= \left(\frac{4ap(p-1)^{p-1}}{(\pi p+1)^{p+1}} \right)^{\frac{1}{4(p-1)}}\|f\|_1\|f\|_2\\
    &=:g_P\|f\|_1\|f\|_2.
\end{align*}
In particular, for $p=2$ we obtain $g_p= \left(\frac{8a}{27\pi}\right)^{1/4}$. For the lower bound it is enough to consider $f(x)=e^{-2a x^{2}}$. In fact, if we consider a pure Gaussian function: $f(x)=e^{-bx^2}$ for some $b>0$, then
\begin{align*}
    \left(\frac{a}{\pi}\right)^{1/2}\int_{\R}\int_{\R}f(x)f(x+t)e^{-a t^{2}}dxdt&=\int_{\R}|\hat f(\xi)|^{2}e^{-\pi^2 \xi^{2}/a}d\xi\\
    &=\frac{\pi}{b}\int_{\R}e^{-\pi^2(2/b+1/a)}\xi^2d\xi\\
    &=\frac{\pi^{1/2}}{(2b+\frac{b^2}{a})^{1/2}}.
\end{align*}
Then
\begin{align*}
    \frac{2^{1/2}\int_{\R}\int_{\R}f(x)f(x+t)e^{-2\pi t^{2}}dxdt}{\|f\|_1\|f\|_2}=\frac{\frac{\pi^{1/2}}{(2b+\frac{b^2}{a})^{1/2}}}{\frac{\pi^{1/2}}{b^{1/2}}\frac{\pi^{1/4}}{(2b)^{1/4}}}=\frac{2^{1/4}}{b^{1/4}\pi^{1/4}(\frac{2}{b}+\frac{1}{a})^{1/2}}.
\end{align*}
Taking derivatives we see that the maximum of the last expression in the right hand side happen when $b=2a$, in that case we obtain
$$
\frac{2^{1/4}}{b^{1/4}\pi^{1/4}(\frac{2}{b}+\frac{1}{a})^{1/2}}=\frac{a^{1/4}}{\pi^{1/4}2^{1/2}}.
$$
In other words $f(x)=e^{-2a x^{2}}$ gives the best lower bound along the family of pure Gausssian functions i.e functions of the form $e^{-bx^2}$.
\end{proof}

\begin{proof}[Proof of Theorem \ref{thm min 1}]
We start proving an auxiliary lemma, this follow from the argument in \cite{BarnardSteinerberger}, we include a short proof of this for completness.\\

\begin{lemma}\label{lemma stefan}
For any $f\in L^{1}(\R)$ the following inequality holds
 \begin{equation}\label{stefan}
\min_{t\in[{-1/2,1/2}]}\int_{\R}f(x)f(x+t)dxdt\leq 0.821534\|f\|^2_1,
\end{equation}
\end{lemma}
\begin{proof}[Proof of Lemma \ref{lemma stefan}]
Consider the function
$$ g(t) :=  \int_{\mathbb{R}}{f(x)f(x+t)dx}.$$
We can assume without loss of generality that
\begin{equation}\label{min=1} \min_{t\in[-1/2,1/2]}g(t) = 1.\end{equation}
By Fubini's theorem, $\int_{\mathbb{R}}{g(t) dt} = \|f\|_{1}^2$. Moreover, there exists a function $p\in L^1(\R)$ such that
$$ g(t) =  \chi_{[-1/2,1/2]}(t) + p(t) \qquad \mbox{and }~p(t) \geq 0\ \text{for all} \ t\in[-1/2,1/2].$$
Then we have that for any $\xi$
\begin{align*}
 0 \leq |\widehat{f}(\xi)|^2 &= \int_{\mathbb{R}}{e^{-2\xi \pi i t} g(t)dt} \\
 &= \int_{\mathbb{R}}{e^{-2\xi \pi i t} \left(   \chi_{[-1/2,1/2]} + p(t) \right) dt} 
 \leq \frac{\sin{(\pi \xi)}}{ \pi \xi}+ \int_{\mathbb{R}}{ p(t)dt}.
 \end{align*}
Therefore
$$\int_{\mathbb{R}}{p(t) dt} \geq  - \inf_{x}\frac{\sin{x}}{x}$$
and we conclude that
\begin{equation}\label{1.21}\|f\|^2_1=\int_{\mathbb{R}}{g(t) dt} \geq  1- \inf_{x}\frac{\sin{x}}{x}=1.217234.\end{equation}
The result follows from \eqref{min=1} and \eqref{1.21}.
\end{proof}
Moreover, using \eqref{HY consequence} with $p=\pi$ we obtain
\begin{align*}
    &\min_{t\in[-1/2,1/2]}\int_{\R}f(x)f(x+t) dx dt\nonumber\\
    &\leq \frac{(2\pi)^{\frac{1}{\pi}}(\pi-1)^{\frac{\pi-1}{2\pi}}}{(\pi+1)^{\frac{\pi+1}{2\pi}}}\left(\int_{\R}\left|\frac{\sin(\pi\xi)}{\pi\xi}\right|^{\pi}d\xi\right)^{\frac{1}{\pi}}\|f\|^{\frac{2}{\pi}}_1\|f\|^{\frac{2\pi-2}{\pi}}_2\\
    &=C_{\pi}\|f\|^{\frac{2}{\pi}}_1\|f\|^{\frac{2\pi-2}{\pi}}_2.
    \end{align*}
Combining this estimate with \eqref{stefan} we conclude that
\begin{align*}
\min_{t\in[{-1/2,1/2}]}\int_{\R}f(x)f(x+t)dxdt &\leq (0.821534^{\pi/2-1}C^{\pi/2}_{\pi})^{1/(\pi-1)}\|f\|_1\|f\|_2\\
&\le 0.829604\|f\|_1\|f\|_2.
\end{align*}
For the lower bound, we consider $f_A(x)=1_{[-A,A]}(x)$. For these functions, we get that $\|f_A\|_1 = 2A, \, \|f_A\|_2 = (2A)^{1/2}.$ A simple calculation also shows that 
\[
\min_{t \in [-1/2,1/2]} \int_{\R} f(t) f(x+t) \, \mmd t = 2A - \frac{1}{2}. 
\]
Therefore, we conclude that $0.829604$ cannot be replaced by 
\[
\sup_{A \ge 1/4} \frac{2A - \frac{1}{2}}{2A(2A)^{1/2}} \sim 0.544.
\]
This finishes the proof.
\end{proof}


\section{Proof of the qualitative results} 

\begin{proof}[Proof of Theorem \ref{Thm 2}]
Let $(f_n)_{n\in\N}\in L^1(\R)\cap L^{2}(\R) \cap \mathcal{L}_G(I)$ be a an extremizer sequence i.e a sequence such that
\begin{equation}
    C_{opt}=\lim_{n\to\infty}\frac{\int_{-1/2}^{1/2}\int_{\R}f_n(x)f_n(x+t) dx dt}{\|f_n\|_1\|f_n\|_2},
\end{equation}
where $G \in L^1(\R), G\ge 0$ is a fixed functions. Via re-scaling we can assume with loss of generality that 
$\|f_n\|_1\|f_n\|_2=1$ for all $n\in \N$.
Then, by Haussdorf-Young inequality and interpolation we have that
\begin{equation}
    \|(\hat f_n)^{2}\|^{1/2}_2=\|\hat f_n\|_4\leq \|f_n\|_{4/3}\leq \|f_n\|_1\|f_n\|_2=1 \ \ \text{for all} \ n\in \N.
\end{equation}
Since $L^{p}(\R)$ is reflexive for all $p>1$, by Banach-Alouglu theorem we conclude that there are functions $M\in L^{4/3}(\R),g\in L^{4}(\R), h\in L^{2}(\R)$ such that
\begin{align}
&f_n \overset{\ast}{\rightharpoonup} M\ \text{in}\ L^{4/3}(\R)\label{wc 1}\\
&\hat f_n \overset{\ast}{\rightharpoonup} g\ \text{in}\ L^{4}(\R)\label{wc 2}\\
\text{and}\ &(\hat f_n)^2 \overset{\ast}{\rightharpoonup} h\ \text{in}\ L^{2}(\R). \label{wc 3}
\end{align}

Thus, clearly $g=\hat M$ a.e since as a consequence of  \eqref{wc 1} and \eqref{wc 2} we have
\begin{align*}
\langle g,J \rangle=\lim_{n\to\infty} \langle \widehat{ f_n},J\rangle=\lim_{n\to\infty}\langle f_n,\widehat{J}\rangle=\langle M,\widehat{J}\rangle\\
=\langle \widehat{M}, J\rangle \ \text{for all}\ J\in L^{4/3}(\R).
\end{align*}

Let $r(x)=\frac{\sin(\pi x)}{\pi x}$ for all $x\in\R$. We observe that as a consequence of \eqref{wc 3}
\begin{align}\label{eq opt direction 1}
&\langle h,r \rangle =\lim_{n\to\infty} \langle(\widehat{ f_n})^{2},r\rangle=\lim_{n\to\infty} \langle \widehat{(f_n*f_n)},r\rangle \nonumber\\
&=\lim_{n\to\infty} \langle f_n*f_n,\chi_{[-1/2,1/2]} \rangle =C_{opt}.
\end{align}

Moreover, notice that the fact that $\{f_n\}_{n \ge 0} \subset \mathcal{L}_G(I)$ and $\widehat{f_n} \overset{\ast}{\rightharpoonup} g$ imply that $\widehat{f_n} \to g$ almost everywhere. Indeed, for each $\varepsilon > 0$ let $I \subset J \subset \R$ be an interval
such that 
$$\int_{J^{c}} \text{max}\{G(x),|M(x)|\} \, \mmd x < \varepsilon.$$
Let $\psi_J,$ on the other hand, be a compactly supported smooth functions with $\psi_J \equiv 1$ on $J,$ $\psi_J \le 1.$ By weak-$*$ convergence, we have
\[
\int f_n(x) \psi_J(x) e^{-2 \pi i x y} \, \mmd x \to \int M(x) \psi_J(x) e^{-2 \pi i x y} \, \mmd x, \, \text{as } n \to \infty, \, \forall y \in \R.
\]
Thus, we have 
\[
|\widehat{f_n}(y) - \widehat{M}(y)| = \left|\int f_n(x) e^{-2 \pi i x y} \, \mmd x - \int M(x) e^{-2 \pi i x y} \, \mmd x \right| 
\]
\[
\le \left| \int f_n(x) \psi_J(x) e^{-2 \pi i x y} \, \mmd x - \int M(x) \psi_J(x) e^{-2 \pi i x y} \, \mmd x\right| + 2 \varepsilon.
\]
Taking $n \to \infty$ then gives us the desired property. As a consequence, we easily conclude that $h = g^2.$ Moreover, for any $J \subset \R$ compact,
$$
\int_{J}|M| = \sup_{\|N\|_{L^{\infty}(J)} = 1} \left| \int_{J} M(x) N(x) \, \mmd x \right| = \sup_{\|N\|_{L^{\infty}(J)} = 1} \lim_{n \to \infty}\left| \int_{J} N(x) f_n(x) \, \mmd x \right|
$$
$$
\le \liminf_{n\to\infty}\int_{\R} |f_n|.
$$
This implies that $\|M\|_{L^1(\R)} \le \liminf_{n \to \infty} \|f_n\|_{L^1}$. Also, by Plancherel, pointwise convergence of $\widehat{f_n}$ to $\widehat{M}$ and Fatou's lemma,
$$
\int_{\R}|M|^{2} = \int_{\R} |\widehat{M}|^2 \leq \liminf_{n\to\infty}\int_{\R} |\widehat{f_n}|^2 = \liminf_{n \to \infty} \int_{\R} |f_n|^2.
$$
Therefore
\begin{align}\label{eq opt direction 2}
\|M\|_{1}\|M\|_{2}&\leq ((\liminf_{n\to\infty}\|f_n\|_1)^{2}\liminf_{n\to\infty} \|f_n\|_{2}^2)^{1/2}\nonumber\\
&\leq \liminf_{n\to\infty} \|f_n\|^{2}_1\|f_n\|^{2}_2\nonumber\\
&=1. 
\end{align}
Combining \eqref{eq opt direction 1} and \eqref{eq opt direction 2} we obtain that $M$ has to be an extremizer. In the end, we only need to verify that $M \not\equiv 0.$ This can be readily seen, for 
instance, from the fact that the left hand side of \eqref{Eq Extremizer 1} only increases under the action of the symmetric decreasing 
rearrangement (by Riesz's rearrangement inequality), whereas the right hand side does not change. Therefore, we can rerun the argument above, now assuming that the sequence $(f_n)_{n \in \N}$ is symmetrically decreasing.
Finally, from \eqref{eq opt direction 1} we see that the limiting function $M,$ which is symmetrically decreasing now, cannot be zero. 

\end{proof}

\begin{proof}[Proof of Theorem \ref{Thm 4}] In order to prove Theorem \ref{Thm 4}, we first observe that, as the class of $L^1$ functions is naturally embedded into the class of positive, 
finite measures, it holds that the constant $C_4$ is less than the constant $\tilde{C}_4$ given by the supremum of the quantity \eqref{eq min measure} over all even, positive, finite measures 
on the real line as in Corollary \ref{Cor 1}.\\ 

\noindent\textit{Step 1: there is $\mu_0$ positive, even, finite measure on the real line such that }
\begin{equation}\label{eq majorize} 
C_4 \le \inf_{1 > t > \eps > 0}\frac{ \mu_0 \star \mu_0 ( [t-\eps,t])}{\eps \|\mu_0\|_{TV}^2}.
\end{equation}

\noindent\textit{Step 1.1.} In order to do it, consider a sequence of functions $\{f_n\}_{n \ge 1}$ extremizing \eqref{eq min}; that is, it holds that 
\[
\min_{t \in [0,1]} \frac{f_n \star f_n(t)}{\|f_n\|_1^2} \to C_4.
\]
Observe that we can assume without loss of generality
that $\|f_n\|_1=1.$
Since $\|f_n\|_1=1,$ the Fourier transforms $\{\widehat{f_n}\}_{n \ge 1}$ of our functions are all bounded by 1 and continuous. It holds, 
in particular, that 
\[
\frac{\widehat{f_n}(x)}{1+|x|} 
\]
are a bounded sequence in $L^2(\R).$ We now use, one more time, the Banach-Alaoglu theorem. This readily implies that we might suppose, after passing to a subsequence, 
that there is a function $h \in L^{\infty}(\R)$ such that 
\[
\frac{\widehat{f_n}(x)}{1+|x|} \overset{\ast}{\rightharpoonup} \frac{h(x)}{1+|x|} \text{ in } L^2(\R).
\]
Indeed, it follows directly from that theorem that there is some $h \in L^2((1+|x|)^2)$ with such properties, and from the definition of weak-$*$ convergence we verify that such $h$ has to be bounded, 
as all $\widehat{f_n}$ are bounded by 1.

\noindent\textit{Step 1.2: identifying $h$ as a Fourier transform of a measure.} An essential step in the proof of existence of $\mu_0$ is to assert that $h$ is the Fourier transform of a non-zero positive measure, which will turn out to be our desired $\mu_0.$ In order to do it, we make use of a classic result by Bochner, identifying functions of positive type and Fourier transforms of positive measures. 

First, we say that a bounded function $\phi : \R \to \C$ is of \emph{positive type} if, for all $g \in L^1(\R),$ it holds that 
\[
\int_{\R\times \R} g(x) \phi(x-y) \overline{g(y)} \, \mmd x \, \mmd y \ge 0.
\]
The following characterization of functions of positive type is usually referred to as \emph{Bochner's theorem}, when one considers instead continuous positive definite functions, and can be found 
as a consequence of Corollary 3.21 and Theorem 4.19 in \cite{Folland}.

\begin{fact}[Bochner's theorem] Let $\phi \in L^{\infty}(\R)$ be a function of positive type on the real line. Then $\phi$ equals almost everywhere the Fourier transform of a positive, finite measure 
$\mu_0$ on $\R.$ 
\end{fact}
We now claim that the function $h$ we constructed above is of positive type. Indeed, by construction, it is bounded from the beginning. As the sequence $\{\widehat{f_n}(x) \cdot (1+|x|)^{-1}\}_{n\ge 1}$ 
converges in the weak-* topology of $L^2,$ we conclude that, whenever $\psi \in C^{\infty}_c(\R),$ we have 
\begin{align}\label{eq limit}
\langle h(x-\cdot), \psi \rangle & = \langle (1+|\cdot|)^{-1}h, \psi(x-\cdot) (1+|\cdot|) \rangle \cr
& = \lim_{n \to \infty} \langle(1+|\cdot|)^{-1} \widehat{f_n}, \psi(x-\cdot) (1+|\cdot|) \rangle. \cr 
\end{align}
By \eqref{eq limit} and the fact that $\|h\|_{\infty},\|\widehat{f_n}\|_{\infty} \le 1,$ it holds that, by dominated convergence, 
\[
\int_{\R \times \R} \psi(x) h(x-y) \overline{\psi(y)} \, \mmd x \, \mmd y = \lim_{n \to \infty} \int_{\R \times \R} \psi(x) \widehat{f_n}(x-y) \overline{\psi(y)} \, \mmd x \, \mmd y \ge 0. 
\]
Therefore, the assertion holds for $\psi \in C^{\infty}_c(\R).$ On the other hand, as $h \in L^{\infty},$ we may use the dominated convergence theorem once more (by approximating an arbitrary $g \in L^1$ 
by $\psi \in C^{\infty}_c(\R)$ in the $L^1$ norm) in order to conclude that 
\[
\int_{\R \times \R} g(x) h(x-y) \overline{g(y)} \, \mmd x \, \mmd y \ge 0, \, \forall \, g \in L^1(\R).
\]
This is exactly the claimed assertion that $h$ is of positive type. By Bochner's theorem, there is $\mu_0$ finite, positive measure such that 
\[
h(\xi) = \widehat{\mu_0}(\xi) \text{ for a. e. } \xi \in \R. 
\]
This $\mu_0$ is our candidate for fulfilling \eqref{eq majorize}. \\

\noindent\textit{Step 1.3: proving $\mu_0$ is our desired measure.} \\

\noindent\textit{(1) $\mu_0 \not\equiv 0.$} Fix $\psi_0 \in C^{\infty}_c(\R)$ so that $1_{J} \le \psi_0 \le 1_{J'}$ for some sufficiently large compact intervals $J' \supset J \supset I$ with $\int_{\R \setminus J} G(x) \, \mmd x < 1/10.$ By weak-* convergence, we have that 
\[
\mu_0(J) \ge \int_{\R} \psi_0(x) \, \mmd \mu_0(x) = \int_{\R} \widehat{\psi_0}(\xi) \widehat{\mu_0}(\xi) \, \mmd \xi = \lim_{n \to \infty} \langle \widehat{\psi_0}, \widehat{f_n} \rangle = \lim_{n \to \infty} \langle \psi_0, f_n \rangle \ge \frac{1}{2}.
\]
This sufffices to prove that $\mu_0 \not\equiv 0.$ \\

\noindent\textit{(2) $\|\mu_0\|_{TV} \le 1.$} This follows from the fact that, for any function $\varphi_0 \in C^{\infty}_c(\R)$ so that $\|\varphi_0\|_{\infty}\le 1,$ it holds that 
\[
\int_{\R} \varphi_0(x) \, \mmd \mu_0(x) = \lim_{n \to \infty} \langle \widehat{\varphi_0}, \widehat{f_n} \rangle \le \|f_n\|_1 = 1. 
\]

\noindent\textit{(3) $\inf_{1 > t > \eps > 0} \frac{1}{\eps} \mu_0 \star \mu_0([t-\eps,t]) \ge C_4.$} Finally, we start by observing the following: 
By the same argument as before, we conclude that there is a function $g \in L^{\infty}(\R)$ such that 
\[
\frac{|\widehat{f_n}(x)|^2}{1+|x|} \overset{\ast}{\rightharpoonup} \frac{g(x)}{1+|x|} \text{ in } L^2(\R).
\]
Now we make fundamental use of the fact that $f_n \in \mathcal{L}_G(I)$ for some $G \ge 0$ in $L^1$ and some compact interval $I.$ Notice that, by the fact that all the functions 
$f_n \in \mathcal{L}_G(I),$ then the sequence of nonnegative measures $\{f_n\}_{n \ge 0}$ satisfies all criteria for \emph{uniform tightness} in the space of nonnegative, finite Borel measures on $\R.$ Therefore, 
we may invoke Prokhorov's theorem to guarantee that, upon passing to a subsequence, for each $\eta \in (C_0 \cap L^{\infty}) (\R),$
\[
\int_{\R} f_n(x) \eta(x) \, \mmd x \to \int_{\R} \eta(x) \, \mmd \nu(x) \text{ as } n \to \infty.
\]
Again, by Fourier uniqueness, it follows that $\nu = \mu_0.$ Now, exactly as in the proof of Theorem \ref{Thm 2}, $\widehat{f_n} \to h(x)$ for all $x \in \R$ follows from 
the relative compactness obtained from Prokhorov's theorem. Thus, we may conclude that $g(x) = |h(x)|^2,$ for almost every $x \in \R.$ \\

\noindent\textit{Step 1.4: Concluding $\mu_0$ is an extremizer over the class we are considering.} We claim that the following two assertions are equivalent:
\begin{enumerate} 
 \item An even function $g$ satisfies $g \star g(t) \ge c$ for almost every $t \in [0,1];$ 
 \item \begin{equation}\label{eq alternative}
\int_{\R} |\widehat{g}(\xi)|^2 \widehat{\varphi}(\xi) \, \mmd \xi \ge c, 
\end{equation}
whenever $\varphi \in C^{\infty}_c(\R)$ is even, positive, compactly supported in $[-1,1],$ with $\int \varphi = 1.$
\end{enumerate}

 In fact, if $g\star g(t) \ge c$ almost everywhere in $[0,1],$ then by evennness of the autocorrelation \eqref{eq alternative} follows. On the other hand, if \eqref{eq alternative} 
holds, we simply pick $\varphi(y) = \frac{1}{2\eps} \left(\psi\left(\frac{x-y}{\eps}\right) + \psi\left(\frac{-x-y}{\eps}\right)\right),$ for $\psi \in C^{\infty}_c(\R),$ positive and even with integral 1. Plugging back into \eqref{eq alternative} and using Plancherel gives 
\begin{equation}\label{eq approx}
\frac{1}{2}((g \star g)*\psi_{\eps}(x) +(g\star g)*\psi_{\eps}(-x)) \ge c,
\end{equation}
where we denote $\psi_{\eps}(z) = \frac{1}{\eps}\psi(z/\eps).$ By the approximate identity theorem, the left hand side of \eqref{eq approx} converges for almost every $x \in [0,1]$ to $g \star g(x).$ This implies that 
\[
\text{ess}\inf_{t \in [0,1]} g\star g(t) \ge c.
\]
We then notice that, as $|\widehat{f_n}(x)|^2/(1+|x|) \overset{\ast}{\rightharpoonup} |h(x)|^2/(1+|x|)$ in $L^2,$ then, for all $\varphi$ as before, 
\[
\int_{\R} \varphi(x) \, \mmd (\mu_0 \star \mu_0)(x) = \int_{\R} |h(\xi)|^2 \widehat{\varphi}(\xi) \, \mmd \xi = \lim_{n \to \infty} \langle (f_n \s f_n), \varphi \rangle \ge C_4. 
\]
Now choose a sequence $\{\varphi_{\delta}\}_{\delta > 0}$ of smooth functions as above, so that 
\[
\frac{1_{[t-\eps,t]} + 1_{[-t,-t+\eps]}}{2(\eps+\delta)} \le \varphi_{\delta} \le \frac{1_{[t-\eps-\delta,t+\delta]} + 1_{[-t-\delta,-t+\eps+\delta]}}{2\eps}.
\]
By the fact that $\mu_0$ is finite, it holds that 
\[
C_4 \le \int_{\R} \varphi_{\delta}(x) \, \mmd (\mu_0 \s \mu_0)(x) \to \frac{1}{\eps}(\mu_0 \s \mu_0)([t-\eps,t]) \text{ as } \delta \to 0.
\]
This proves the third assertion, and thus also \eqref{eq majorize}. \\

\noindent\textit{Step 2: we have that} 
\[
C(\mu_0) := \inf_{1 > t > \eps > 0} \frac{(\mu_0 \s \mu_0)([t-\eps,t])}{\eps \|\mu_0\|_{TV}^2} < \frac{1}{2(1+\theta_0)}.
\]

In order to do this, we divide again into several steps:\\

\noindent\textit{Step 2.1: We redo the proof by Barnard and Steinerberger of \cite[Theorem~2]{BarnardSteinerberger}, also partially exposed in the proof of Theorem \ref{thm min 1}.}\\

Indeed, by normalizing the measure $\mu_0$ differently, we might suppose that 
\[
\inf_{1 > t > \eps > 0} \frac{(\mu_0 \s \mu_0)([t-\eps,t])}{\eps} = \frac{1}{2}.
\]
This implies, in particular, that the measure 
\[
\mmd \nu := \mmd (\mu_0 \s \mu_0) - \frac{1}{2}1_{[-1,1]}(x) \, \mmd x
\]
is nonnegative (as the measure of any closed interval is $\ge 0$). We now take Fourier transforms: 
\begin{align}\label{eq Fourier estimate}
0 \le |h(\xi)|^2 & = \mathcal{F}(\mu_0 \s \mu_0) (\xi) = \mathcal{F}(\frac{1}{2}1_{[-1,1]})(\xi) + \mathcal{F}(\nu)(\xi) = \frac{\sin(2\pi \xi)}{(2\pi \xi)} + \widehat{\nu}(\xi). 
\end{align}
Let then $\xi_0$ be the point where $\xi \mapsto \frac{\sin(2\pi \xi)}{2\pi \xi}$ attains its global minimum; that is, $\xi_0 = \frac{y_0}{2\pi}.$ We have, from \eqref{eq Fourier estimate},
\[
\widehat{\nu}(\xi_0) \ge \theta_0.
\]
Now we wish to show that $\widehat{\nu}(\xi_0) < \widehat{\nu}(0) = \|\mu_0\|_{TV}^2 - 1,$ which would finish the proof of Theorem \ref{Thm 4}. Let us suppose, therefore, that the strict 
inequality does not hold. As $\mmd \nu$ is a \emph{positive}, even measure, there must hold thus equality: 
\[
\widehat{\nu}(\xi_0) = \widehat{\nu}(0) \iff \int_{\R}(1-\cos(2 \pi \xi_0 t)) \, \mmd \nu(t) = 0.
\]
As $1- \cos(2\pi \xi_0 t) > 0$ if $t \not \in \mathbb{Z}/\xi_0$ and $\mmd \nu(t) \ge 0,$ we conclude that $\nu(\R\backslash(\Z/\xi_0)) = 0.$ This implies, in particular, that 
\begin{equation}\label{eq delta}
\mmd(\mu_0 \s \mu_0) - \frac{1}{2} 1_{[-1,1]} = \mmd \nu  = \sum_{i \ge 0} a_i (\delta_{i/\xi_0} + \delta_{-i/\xi_0}),
\end{equation}
for some sequence $\{a_i\}_{i \ge 0}$ of nonnegative numbers.\\ 

\noindent\textit{Step 2.2: analysis of measures satisfying \eqref{eq delta}} \\

\noindent Indeed, let us suppose \eqref{eq delta} holds, and let $\mmd \mu_0 = \mmd \mu_{pp} + \mmd \mu_{sc} + f_0(x) \, \mmd x$ be the Lebesgue-Radon-Nykodim decomposition of $\mu_0,$ where 
$\mu_{pp}$ is a discrete measure, and $\mu_{sc}$ is singular continuous. Equation \eqref{eq delta} then translates as 
\[
2 (\mmd (\mu_{pp} \s \mu_{sc}) + \mmd \mu_{pp} \s f_0 \, \mmd x + \mmd \mu_{sc} \s f_0 \, \mmd x) + (f_0 \s f_0) \, \mmd x + \mmd (\mu_{pp} \s \mu_{pp}) + \mmd (\mu_{sc} \s \mu_{sc}) =
\]
\begin{equation}\label{eq expand}
=  \frac{1}{2} 1_{[-1,1]} + \sum_{i \ge 0} a_i (\delta_{i/\xi_0} + \delta_{-i/\xi_0}). 
\end{equation}
In order to better understand this equality, we must first understand the interactions between measures of different
nature when convolved. This is the content of the following Lemma. 

\begin{lemma}\label{lemma convolution} Let $\mu,\nu$ be two finite, positive measures on $\R$. Then the following assertions hold: 
\begin{enumerate}
 \item[(i)] If either $\mu$ or $\nu$ is non-atomic, then so is $\mu \s \nu;$
 \item[(ii)] If either $\mu$ or $\nu$ is absolutely continuous, so is $\mu \s \nu$ 
\end{enumerate}
\end{lemma}

\begin{proof} 

\noindent\textit{Proof of (i).} Without loss of generality, let $\mu$ be non-atomic. That is, for each $x \in \R, \, \mu(\{x\})=0.$ We simply compute
\begin{align*}
\mu \s \nu(\{z\}) = \int_{\R} \left( \int_{\R} 1_{\{z\}}(x-y) \, \mmd \mu(x) \right)\, \mmd \nu(y) = \int_{\R} \mu(\{z+y\}) \, \mmd \nu(y) =0,
\end{align*}
which is a direct consequence of Fubini's theorem. \\

\noindent\textit{Proof of (ii).} In the same way, let $\mu$ be absolutely continuous. Let $A \subset \R$ be so that $|A| = 0.$ Then 
\begin{align*}
\mu \s \nu(A) = \int_{\R} \left(\int_{\R} 1_{A}(x-y) \, \mmd \mu(x)\right) \, \mmd \nu(y) = \int_{\R} \mu(\tau_{-y}(A)) \, \mmd \nu (y) = 0,
\end{align*}
again by Fubini's theorem, where $\tau_y(B) = \{z \in \R \colon z+ y \in B\}.$ 
\end{proof}

From Lemma \ref{lemma convolution}, it holds that all the discrete part of the autocorrelation $\mu_0 \s \mu_0$ must coincide with $\mu_{pp} \s \mu_{pp}.$ In other words, 
\[
\mmd(\mu_{pp} \s \mu_{pp}) = \sum_{i \ge 0} a_i (\delta_{-i/\xi_0} + \delta_{i/\xi_0}).
\]
Let $x_0,y \in \R$ be two points such that $\mu_{pp}(\{y\}),\mu_{pp}(\{x_0\}) > 0.$ Then $\mu_{pp} \s \mu_{pp}(\{y - x_0\}) > 0,$ which implies directly 
that 
\[
\mu_{pp} = \sum_{i \in \Z} b_i \delta_{i/\xi_0 + x_0}.
\]
By relabelling the indices, we may assume that $x_0 \in (0,1/\xi_0).$ Equation \eqref{eq delta} yields yet another consequence: noticing that the measure $\mu_0 \s \mu_0$ coincides with $\mu_{pp} \s \mu_{pp}$
outside the interval $[-1,1],$ we compute:
\begin{align}\label{eq contra} 
0 &= \int_{\R\backslash[-1-\eps,1+\eps]} \mmd(\mu_0 \s \mu_0 - \mu_{pp} \s \mu_{pp})(t) \cr 
  &\ge 2 \sum_{j \ge 0} b_j \int_{\R \backslash [-1-\eps,1+\eps]} \delta_{j/\xi_0 + x_0}\s f_0(t) \, \mmd t \cr 
  &\ge 2 \sum_{j \ge 0} b_j \left(\int_{1+\eps}^{\infty}f_0(t+(j/\xi_0 + x_0)) \, \mmd t + \int_{-\infty}^{-1-\eps} f_0(t+(j/\xi_0+ x_0)) \, \mmd t\right). \cr 
\end{align}
Suppose now there are $i,j \in \Z$ so that $b_i,b_j > 0.$ The last lower bound in \eqref{eq contra} is at least as large as 
\[
\min\{b_i,b_j\} \int_{\R \backslash [-1-\eps,1+\eps]} (f_0(x + (i/\xi_0 + x_0)) + f_0(x + (j/\xi_0 + x_0))) \, \mmd x.
\]
This last display is, in turn, at least 
\begin{equation}\label{eq lower}
\min\{b_i,b_j\} \int_{A_{i,j,\eps}} f_0(x) \, \mmd x,
\end{equation}
where 
$$A_{i,j,\eps} = \R \backslash ([-1-\eps + (i/\xi_0 + x_0), 1 + \eps + (i/\xi_0 + x_0)] \cap [-1-\eps+ (j/\xi_0 + x_0),1+\eps + (j/\xi_0 + x_0)]).$$ 
We have therefore that \eqref{eq lower} is at least $\min\{b_i,b_j\} \|f_0\|_1,$ in case $|i-j| > (2+\eps) \xi_0.$ As $\xi_0  < 0.75$ and $\eps>0$ is arbitrary, we conclude that either $f_0 \equiv 0,$ or the measure 
$\mu_{pp}$ is supported on two points $i_0, i_0 + 1.$ \\

\noindent\textit{Step 2.3: Conclusion.} \\

\noindent\textit{Case 1: $\mu_{pp} \equiv 0.$} This is the simpler case. Indeed, \eqref{eq expand} simplifies to 
\[
\mmd (\mu_0 \s \mu_0) = \frac{1}{2} 1_{[-1,1]}.
\]
Taking Fourier transforms of both measures yields a simple contradiction, as the Fourier transform of autocorrelations is always nonnegative, whereas the Fourier transform of the (normalized) characteristic function 
of $[-1,1]$ is $\sin(2\pi \xi)/(2\pi \xi).$ \\

\noindent\textit{Case 2a: $\mu_{pp} \not\equiv 0, \, f_0 \equiv 0.$} In this case, \eqref{eq expand} becomes, after cancelling out the atomic parts,
\begin{equation}\label{eq simplified}
\frac{1}{2} 1_{[-1,1]} = 2 \mmd (\mu_{pp} \s \mu_{sc}) + \mmd (\mu_{sc} \s \mu_{sc}).
\end{equation}
If $A' \subset \R$ is a measurable set so that $\mu_{sc} (A') > 0, |A'| = 0,$ then the measure on the right hand side of \eqref{eq simplified} of $\tau_{-x_0}(A')$ is positive, whereas evaluating the absolutely continuous
measure on its left hand side to the same set yields 0, a contradiction. Therefore, $\mu_{sc} \equiv 0.$ But this leads to an automatic contradiction in \eqref{eq simplified}. \\

\noindent\textit{Case 2b: $\mu_{pp} \not \equiv 0, \, f_0 \not\equiv 0.$} This is the main case. In analogy to \textbf{Case 2a}, we can argue once again with \eqref{eq expand} in conjunction with Lemma \ref{lemma convolution}
to obtain that $\mu_{sc} \equiv 0.$ We will skip the details, as they are essentially the same to \eqref{eq simplified} and the considerations thereafter. 

We have, thus, that $\mu_{sc} \equiv 0.$ We write $\mu_{pp} = a \mu_{i_0/\xi_0 + x_0} + b \mu_{(i_0 + 1)/\xi_0 + x_0}, \, a,b > 0.$ 
Equation \eqref{eq expand} then becomes, after cancelling out the atomic parts, 
\begin{equation}\label{eq conv}
2a f_0(x + i_0/\xi_0 + x_0) + 2b f_0(x + (i_0+1)/\xi_0 + x_0) + f_0 \s f_0(x) = \frac{1}{2} 1_{[-1,1]}(x),
\end{equation}
for almost all $x \in \R.$ \\

\noindent\textit{Case 2ba: $a > 0 = b$ in \eqref{eq conv}} In this case, we notice that letting $g_0(x) = f_0(x + i_0/\xi_0 + x_0)$ implies that $g_0$ is a solution to the following equation:

\begin{equation}\label{eq conv2} 
\frac{1}{2} 1_{[-1,1]} = 2a g_0  + g_0\s g_0.
\end{equation}

\begin{claim} There is no positive, integrable solution $g_0$ to \eqref{eq conv2}. 
\end{claim}

\begin{proof} It follows directly that $\text{supp}(g_0) \subset [-1,1]$ and $g_0 \in L^2(\R).$ Therefore, by the Paley-Wiener theorem, $\widehat{g_0}$ is a function of exponential type $\sigma \le 2 \pi.$ 
By taking Fourier transforms of \eqref{eq conv2}, we obtain that $|\widehat{g_0}|^2$ is a function of exponential type $\le 2 \pi$ as well, which implies that $\widehat{g_0}$ is a function of exponential 
type $\sigma \le \pi.$ By the converse of the Paley-Wiener theorem, $\text{supp}(g_0) \subset [-\frac{1}{2},\frac{1}{2}].$ 

On the other hand, taking the limit of $t \to 1_{-}$ of 
\begin{equation}\label{eq conv21}
\frac{1}{2}1_{[-1,1]} - 2a g_0 = g_0 \s g_0,
\end{equation}
we obtain that the left hand side converges to $\frac{1}{2},$ as $\text{supp}(g_0) \subset [-\frac{1}{2},\frac{1}{2}].$ On the other hand, as $g_0 \in L^2(\R),$ the convolution 
$g_0 \s g_0$ is \emph{continuous}, and has compact support in $[-1,1].$ Thus, the limit as $t \to 1_{-}$ of the right hand side is $0,$ a contradiction. 
\end{proof}

\noindent\textit{Case 2bb: $a, b > 0$ in \eqref{eq conv}} In this case, again by translating $f_0$ by $i_0/\xi_0 - 1/(2\xi_0) + x_0,$ we end up with the task of solving 

\begin{equation}\label{eq conv3}
\frac{1}{2} 1_{[-1,1]}(t) = 2bf_0(t-1/(2\xi_0)) + 2af_0(t+1/(2\xi_0))) + f_0 \s f_0(t).
\end{equation}

\begin{claim} There is no positive, integrable solution $f_0$ to \eqref{eq conv3} 
\end{claim}

\begin{proof} It follows now that $\text{supp}(f_0) \subset [-1+\alpha_0,1-\alpha_0],$ where $\alpha_0 = \frac{1}{2\xi_0},$ 
and again $f_0 \in L^2(\R).$ Therefore, $\text{supp}(f_0 \s f_0) \subset [-2(1-\alpha_0),2(1-\alpha_0)].$ But then it follows from \eqref{eq conv3} that 
\[
f_0(t-\alpha_0) = \frac{1}{4b} \text{ on } [-1,-2(1-\alpha_0)],
\]
which implies that $f \equiv \frac{1}{4b}$ on $[-1+\alpha_0, -2+3\alpha_0] \supset [-1+\alpha_0,0],$ as $\alpha_0 > 2/3 \iff \xi_0 < 0.75,$ which is true, as $\xi_0 = 0.71514\dots.$ By the same token applied 
to $f_0(t+\alpha_0)$ on $[2(1-\alpha_0),1],$ we have that
\[
f_0(t) = \frac{1}{4a}1_{[2-3\alpha_0,1-\alpha_0]}(t).
\]
This promptly implies that $a = b$ and 
\begin{equation}\label{eq charact}
f_0(t) = \frac{1}{4a} 1_{[-1+\alpha_0,1-\alpha_0]}(t)
\end{equation}
is the only possible solution fulfilling our requirements. An easy computation substituting \eqref{eq charact} into \eqref{eq conv3} shows that this is not a solution, and therefore finishes our proof.
\end{proof}

After this careful case analysis, we see that no positive measure can satisfy \eqref{eq delta}. In particular, we have that the \emph{strict} inequality 
\[
\widehat{\nu}(\xi_0) < \widehat{\nu}(0)
\]
must hold, which shows that the best constant for \eqref{eq min} is strictly less than $\frac{1}{2(1+\theta_0)}.$ That is what we wished to prove. 
\end{proof} 

Finally, we employ the ideas in the proof of Theorem \ref{Thm 4} to prove Corollary \ref{Thm 5}.

\begin{proof}[Proof of Corollary \ref{Thm 5}] Let $\{\mu_n\}_{n \in \N}$ be an extremizing sequence for \eqref{eq min measure}.
The sequence $\{\widehat{\mu_n}(x)/(1+|x|)\}_{n \in \N}$ is again a bounded sequence in $L^2(\R),$ and therefore once more by the Banach-Alaoglu theorem, we may 
extract a weak-$*$ convergent subsequent, and this we assume without loss of generality that the sequence itself is convergent. By the considerations in the proof of Theorem \ref{Thm 4}, we see that there
is a function $w \in L^{\infty}(\R)$ so that 
\[
\frac{\widehat{\mu_n}(\xi)}{(1+|\xi|)} \overset{\ast}{\rightharpoonup} \frac{w(\xi)}{(1+|\xi|)}.
\]
We now claim that the function $w$ is of positive type. Indeed, this follows almost verbatim the argument in \eqref{eq limit}, and so we skip the argument. By Bochner's theorem one more time, 
we see that $w = \widehat{\tilde{\mu}},$ for $\tilde{\mu}$ a nonnegative measure. By the fact that $\mu_n(J) \ge \frac{1}{2}$ for a prefixed $J \subset \R,$ we have that $\tilde{\mu}(J) \ge \frac{1}{4},$ and therefore 
it is not the zero measure. Also, it is direct from the definition that $\|\tilde{\mu}\|_{TV} \le 1$, and, from the alternative characterization \eqref{eq alternative} of our minimization problem (and the compactness techniques employed in the proof of Theorem \ref{Thm 4}), 
we see that $\inf_{0 < \eps < t < 1} \frac{1}{\eps} \tilde{\mu} \s \tilde{\mu}([t-\eps,t]) \ge \tilde{C}_4.$ But the definition of this constant implies that 
$\tilde{\mu}$ is, in fact, an extremizer to \eqref{eq min measure} within our restricted class, as desired.
\end{proof}

\section{Comments and Remarks} 

\subsection{The dual formulation of Problem \eqref{eq min} and lower bounds} As previously mentioned in the text, we may rephrase the problem of finding the minimal constant $C_4$ such that the inequality 
\begin{equation}\label{eq min 2}
\inf_{t \in [0,1]} f \s f(t) \le C_4 \|f\|_1^2
\end{equation}
in terms of a dual problem on the Fourier side \eqref{eq alternative}: normalizing $\|f\|_1 = 1,$ then a function $f \in L^1(\R)$ extremizes \eqref{eq min 2} if and only if, for each even, positive
function $\varphi \in C^{\infty}_c(\R)$ supported in $[-1,1]$ with integral $1$, 
\begin{equation}\label{eq alternative 2} 
\int_{\R} |\widehat{f}(\xi)|^2 \widehat{\varphi}(\xi) \, \mmd \xi \ge C_4. 
\end{equation}
After investigating this equation, one is tempted to try to improve the bound $C_4 \le \frac{1}{2(1+\theta_0)}$ by working with this dual problem instead. In particular, we see from \eqref{eq alternative 2} 
that 
\[
\|(\widehat{\varphi})_+\|_1 \ge C_4,
\]
and therefore any upper bound on the value of $\|(\widehat{\varphi})_+\|_1$ yields an automatic upper bound on $C_4.$ Here and henceforth, we denote $\max(0,f) = f_+, \, \max(0,-f) = f_-.$The next result shows, however, that this attempt does not give us \emph{any} improvement 
over the original result of \cite{BarnardSteinerberger}. 

\begin{theorem}\label{thm lower} Let $\varphi$ be a smooth function as above. It holds that 
\[
\|(\widehat{\varphi})_+\|_1 \ge \frac{1}{2(1+\theta_0)}.
\]
\end{theorem}

\begin{proof} The proof of this result resembles, in spirit, the method of Bourgain, Clozel and Kahane \cite{BCK} to provide lower bounds for the root uncertainty principle in any dimension. 

Indeed, we know that
\begin{align}\label{eq sum diff}
\varphi(x) \le \|\widehat{\varphi}\|_1 & = \|(\widehat{\varphi})_+\|_1 + \|(\widehat{\varphi})_-\|_1, \, \forall x \in [-1,1],\cr 
\varphi(0) & = \|(\widehat{\varphi})_+\|_1-\|(\widehat{\varphi})_-\|_1. \cr 
\end{align}
Also, we know that 
\begin{align}\label{eq bound negative}
1 - 2 \varphi(0) = \int_{-1}^1 (\varphi(t) - \varphi(0)) \, \mmd t & = \int_{-1}^1 \left(\int_{\R}\widehat{\varphi}(\xi)(\cos(2\pi \xi t ) - 1)\, \mmd \xi\right) \, \mmd t \cr 
 & \le \int_{-1}^1 \left(\int_{\R} (\widehat{\varphi})_-(\xi) (1-\cos(2\pi \xi t)) \, \mmd \xi \right) \, \mmd t \cr 
 & \le 2(1+\theta_0) \|(\widehat{\varphi})_-\|_1,
\end{align}
so using \eqref{eq bound negative} in the second equation in \eqref{eq sum diff} implies that 
\[
\|(\widehat{\varphi})_+\|_1 \ge \frac{1}{2(1+\theta_0)} + \frac{\theta_0}{1+\theta_0}\varphi(0) \ge \frac{1}{2(1+\theta_0)},
\]
as desired. 
\end{proof}

We notice that this idea does not only work in dimension one. Indeed, if one adapts the proof of Theorem \ref{thm lower} for the higher dimensional case, one obtains an asymptotic growth resembling that of Bourgain, Clozel and Kahane \cite[Théor\'eme~3]{BCK} for
the value of 
\[
\mathbb{A}_d = \inf_{f \in \mathcal{A}_d\backslash\{0\}} A(f)A(\widehat{f}),
\]
where $\mathcal{A}_d$ denotes the class of even, real and integrable functions $f \in \R^d$ whose Fourier transform share the same properties, together with $f(0), \widehat{f}(0) \le 0,$ and 
$$
A(g) = \inf\{r > 0 \colon f(x) \ge 0, \, \forall |x| \ge r\}.
$$
See, for instance, \cite{CG,GOeSS, GOeSR1, GOeSR2} for sharper estimates and more recent developments in the study of the constants $\mathbb{A}_d.$ 

We do not know of any dimension for which we can find the sharp form of the lower bound in Theorem \ref{thm lower}, but we believe that there is a strong connection between this new problem and the framework 
of problems we just mentioned. 

\subsection{The compactly supported version of Problem \eqref{eq min} and equivalences} In \cite{BarnardSteinerberger}, the authors prove that the best constant $C_4$ for the inequality \eqref{eq min} above is at least $0.37$. This is due to an explicit counterexample:
if one lets 
\[
f(x) = \frac{1_[-1/2,1/2]}{\sqrt{1 - 4x^2}} - \frac{1_{[-1/4,1/4]}}{4\sqrt{1-4x^2}},
\]
then $f \s f(t) \ge \frac{\pi}{4}$ for $t \in [0,1],$ while $\|f\|_1 \le 1.439.$ The fact that Barnard and Steinerberger manage to come so close 
to the upper bound with a relatively simple compactly supported example leads us to the following conjecture. 

\begin{conjecture}\label{conj compact} It holds that 
$$
C_4 = \sup_{f \in L^1([-1/2,1/2])} \inf_{t \in [0,1]} \frac{|f \s f(t)|}{\|f\|_1^2}
$$
\end{conjecture}
Regarding this conjecture, we have the following partial progress: 

\begin{proposition} It holds that
$$
C_4 \le 4 \sup_{f \in L^1([-1,1])} \inf_{t \in [0,1]} \frac{|f \s f(t)|}{\|f\|_1^2}
$$
\end{proposition}

\begin{proof} Fix $\eps > 0.$ We let $g$ be a function so that 
\[
|g \s g(t)| \ge (C_4 - \eps)\|g\|_1^2, \, \forall t \in [0,1].
\]
We define then 
\[
G(x) = 1_{[-1,1]}(x) \left( \sum_{n \in \Z} g(x-n)\right). 
\]
A straightforward computation shows that 
\begin{align*}
G\s G(t) = \int_{\R} G(x) G(x+t) \, \mmd x & = \int_{-1+t}^1 \left(\sum_{n \in \Z} g(x-n) \right)\left(\sum_{m \in \Z} g(x+t-m)\right) \, \mmd x \cr 
 & = \sum_{m,n \in \Z} \int_{n-1}^{n+1-t} g(x)g(x+t + (n-m)) \, \mmd x \cr
 & \ge \sum_{n \in \Z} \int_{n-1}^{n+1-t} g(x)g(x+t) \, \mmd x \ge g \s g(t).
\end{align*}
However, we know that $\|G\|_1 = 2 \|g\|_1,$ as $G$ is $1-$periodic and its restriction to $[-1/2,1/2]$ has the same integral as $g.$ Therefore, 
\[
|G \s G(t)| \ge \frac{1}{4}(C_4 - \eps)\|G\|_1^2, \forall t \in [0,1].
\]
Taking the supremum of this last expression over all $G \in L^1([-1,1])$ and letting $\eps \to 0$ finishes the proof. 
\end{proof}

As $\frac{0.42}{0.37} < 4,$ this proof is redundant when the purpose is improving the constant for comparison between the two theorems. 
We believe, however, that this proof has more to offer beyond this (very raw) comparison principle. In fact, we currently believe that taking the 
normalized function 
\[
\tilde{G}(x) = 1_{[-1/2,1/2]}(x) \left(\sum_{n \in \Z} g(x-n)\right)
\]
might prove this inequality. This is in accordance to the fact that the Fourier-dual version of the problem \eqref{eq alternative 2} is the same as demanding 
that 
\[
\int_{\R} P_{2\pi}(|\widehat{f}|^2)(\xi) \widehat{\varphi}(\xi) \, \mmd \xi \ge C_4,
\]
where $P_{2\pi}(g)$ denotes the projection of the function $g \in L^2(\R)$ onto the Paley--Wiener space 
\[
PW_{2\pi}(\R) = \{ h \in L^2(\R) \colon \widehat{h} \subset [-1,1]\}.
\]
We believe that if one can relate the positivity of $|\widehat{f}|^2$ to positivity of $P_{2\pi}(|\widehat{f}|^2),$ it may be possible to prove 
Conjecture \ref{conj compact} true. 

\subsection{Smooth approximations} Theorem \ref{Thm 4} and most of our results deal with the problem of bounding the minimum of the autocorrelations 
\[
f \s f(t) = \int_{\R} f(t) f(x+t) \, \mmd t 
\]
on the interval $[0,1].$ The extremal function for this inequality might not be attained for a very smooth function. Indeed, our proof above only gives a positive measure attaining
extremality. On the other hand, we could have narrowed down our search to the class of $C^{\infty}$ functions from the beginning. 
To that extent, define $K_4$ to be the smallest constant so that the equation  
\begin{equation}\label{eq conv constant} 
\min_{t \in [0,1]} f\s f(t) \le K \|f\|_1^2
\end{equation}
holds for all positive functions in $L^1(\R) \cap C^{\infty}(\R).$ In this regard, we can in fact prove that the two problems are \emph{equal}.

\begin{theorem}\label{Thm last} With the previous notation, it holds that $C_4 = K_4.$
\end{theorem}

\begin{proof} Pick $f \in L^1(\R)$ positive so that 
\begin{equation}\label{eq almost extreme}
\min_{x \in [0,1]} f \s f(x) \ge (C_4 -\eps) \|f\|_1^2.
\end{equation}
Let $f_{\lambda}(x) = f(\lambda x).$ Then we see that $\|f_{\lambda}\|_1 = \frac{1}{\lambda}\|f\|_1,$ whereas 
$f_{\lambda} \s f_{\lambda} (x) = \frac{1}{\lambda} (f\s f)(\lambda x).$ Therefore, \eqref{eq almost extreme} implies in particular that 
\[
\min_{x \in [0,1/\lambda]} f_{\lambda} \s f_{\lambda} (x) \ge \lambda (C_4 - \eps) \|f_{\lambda}\|_1^2.
\]
Therefore, if $\lambda < 1$ is sufficiently close to $1,$ we can ensure that 
\begin{equation}\label{eq dilation}
\min_{x \in [0,1/\lambda]} f_{\lambda} \s f_{\lambda} (x) \ge (C_4 - 2\eps) \|f_{\lambda}\|_1^2.
\end{equation}
Let $\psi \in C^{\infty}_c(\R)$ be a smooth, positive, even function supported in $[-1,1]$ with integral 1, and let $\psi_t(x) = \frac{1}{t} \psi\left(\frac{x}{t}\right)$ be the associated approximate identity 
family. Fixing some $t < \frac{1}{2} \left(\frac{1}{\lambda} - 1\right),$ we consider 
\[
\tilde{f}(x) = f_{\lambda} * \psi_t (x).
\]
A simple computation shows that $\tilde{f} \s \tilde{f} = (f_{\lambda} \s f_{\lambda}) * (\psi_t * \psi_t).$ As $\text{supp}(\psi_t * \psi_t) \subset [1 - 1/\lambda,1/\lambda - 1],$ 
\eqref{eq dilation} shows us that $\tilde{f} \s \tilde{f}(x) \ge (C_4 - 2\eps)\|f_{\lambda}\|_1^2$ for $x \in [0,1].$ But $\|\tilde{f}\|_1 = \|f_{\lambda}\|_1$ by the definition of $\psi,$ so that 
\begin{equation}\label{eq smooth}
\min_{x \in [0,1]} \tilde{f} \s \tilde{f} (x) \ge (C_4 - 2 \eps) \|\tilde{f}\|_1^2.
\end{equation}
This finishes our proof.
\end{proof}

Notice that the proof above shows us that, with an additional approximation argument, we can even suppose that $f$ has compact support. We believe that the equivalence between these definitions of our extremal problem 
might be helpful when searching for extremal functions and running numerical methods. 

\section*{Acknowledgements} 

We would like to thank Stefan Steinerberger, who told the second author about the results in \cite{BarnardSteinerberger} and shared ideas on the topic. The first author is thankful to Terence Tao and Oscar Madrid-Padilla for suggestions and interesting discussions during the preparation of this work. Part of this work was accomplished during a visit
of the second author to the mathematics department of the University of California, Los Angeles, which he thanks for hospitality. Finally, both authors are thankful for the comments of the anonymous referees.

\end{document}